\documentclass[15pt]{amsart}
\usepackage[all,cmtip]{xy}
\usepackage{amsthm}
\usepackage{xypic}
\usepackage{graphicx}
\usepackage{amscd}

\usepackage[colorlinks=false, linkcolor=red]{hyperref}

\newcommand{\M}{\overline{\mathcal{M}}}

\newcommand{\Mor}{\textbf{Mor}}
\newcommand{\Fano}{\mathrm{Fano}}

\newcommand{\C}{\mathbb{C}}
\newcommand{\CH}{\mathrm{CH}}
\newcommand{\II}{\mathcal{I}}
\newcommand{\im}{\textbf{Im}}

\newcommand{\PP}{\mathbb{P}}

\newcommand{\OO}{\mathcal{O}}

\newtheorem{theorem}{Theorem}[section]
\newtheorem{definition}{Definition}[section]

\newtheorem{cor}[theorem]{Corollary}
\newtheorem{lemm}[theorem]{Lemma}
\newtheorem{prop}[theorem]{Proposition}
\newtheorem{propo}[theorem]{Proposition}

\newtheorem{question}[theorem]{Question}

\newtheorem{example}[theorem]{Example}

\DeclareMathOperator{\Fa}{F}
\DeclareMathOperator{\Spec}{Spec}
\DeclareMathOperator{\AJ}{AJ}
\DeclareMathOperator{\Span}{Span}
\DeclareMathOperator{\pr}{Pr}
\DeclareMathOperator{\Fl}{Fl}
\DeclareMathOperator{\Ho}{H}

\DeclareMathOperator{\Sym}{Sym}

\numberwithin{equation}{theorem}
\title []{2-Cycles on Higher Fano Hypersurfaces}

\author {xuanyu pan}
\address{Department of Mathematics, Washington University in St.Louis, St.Louis, MO 63130}
\email{pan@math.wustl.edu}
\date{\today}
\begin{document}
\begin{abstract}
Let $X_d$ be a smooth hypersurface of degree $d$ in $\PP_{\C}^n$. Suppose that the Fano variety $\mathrm{F(X_d)}$ of lines of $X_d$ is smooth. We prove that the Griffiths group $\mathrm{Griff}_1(\mathrm{F}(\mathrm{X_d}))$ of $\mathrm{F(X_d)}$ is trivial if the hypersurface $\mathrm{X_d}$ is of 2-Fano type. As a result, we give a positive answer to a question of Professor Voisin about the first Griffiths groups of Fano varieties in some cases. Base on this result, we prove that $\CH_2(\mathrm{X_d})=\mathbb{Z}$ for a smooth $3$-Fano hypersurface $\mathrm{X_d}\subseteq \PP^n_{\mathbb{C}}$ with smooth Fano variety of lines. 
\end{abstract}

\maketitle
\tableofcontents
\section{Introduction}
A fundamental question about cycles on a smooth projective variety $X$ over complex numbers is to determine the Griffiths groups of $X$. Let us recall the definition of the first Griffiths group \[\mathrm{Griff}_1(X)=\frac{\CH_1(X)_{hom}}{\CH_1(X)_{alg}}.\] Professor Voisin asks the following question:
\begin{question}\label{q1}
Is the $\mathrm{Griffiths}$ group  $\mathrm{Griff}_1(X)$ of a Fano variety $X$ (or more generally, a rationally connected variety) over complex numbers  trivial?
\end{question}
In general, to answer this question is very difficult. However, in the case of dimension at most three, Bloch and Srinivas give a positive answer, see \cite{BS}. Recently, Tian and Zong give a positive answer for Fano complete intersections in a projective space, see \cite{TZ}.
Another fundamental question about cycles is to determine the Chow groups of a variety. The geometry of Chow groups is very delicate. For instance,  Mumford proves that the Chow group of zero cycles $\CH_0(S)$ of a K3 surface $S$ is infinitely dimensional, see \cite{mumford}. For a Fano variety $X$ over complex numbers, J\'anos Koll\'ar, Yoichi Miyaoka, and Shigefumi Mori prove that \[\CH_0(X)=\mathbb{Z}\]since $X$ is rationally connected, see \cite{KMM}. In the paper \cite{HF}, de Jong and Starr introduce a concept of higher Fano varieties. Typical examples of higher Fano varieties are low degree hypersurfaces. In the paper \cite{TZ}, Tian and Zong prove that \[\CH_1(X_d)=\mathbb{Z}\] for a smooth hypersurface $X_d$ of 2-Fano type, see \cite{TZ} for details. In this paper, we use recent techniques from the rational curves on algebraic varieties due to de Jong and Starr, Harris and Roth (cf. \cite{DS},\cite{S1} and \cite{S2}) to give a positive answer to the question (\ref{q1}) in the case of Fano varieties of lines. As a result, we prove that $\CH_2(X_d)=\mathbb{Z}$ for a general smooth $3$-Fano hypersurface $X_d$ over complex numbers. More precisely, we have
\begin{theorem}\label{mainthmab}
Let $X_d$ be a hypersurface of degree $d\geq 3$ in $\PP^n_{\mathbb{C}}$.
\begin{enumerate}
\item Suppose that the Fano variety $\Fa(X_d)$ of lines is smooth. If $d\geq 4$ and \[n\geq 3\binom {d+1}{2}-d-4,\]or $d=3$ and $n\geq 14$, then $\mathrm{Griff}_1(\Fa(X_d))=0.$
\item  If the Fano variety $\Fa(X_d)$ of lines is smooth and \[\frac{d(2d^2+1)}{3}\leq n,\] then $\CH_2(X_d)=\mathbb{Z}$.
\end{enumerate}
\end{theorem}
We know the assumption of the smoothness of $\Fa(X_d)$ always holds if $d=3$ or $X_d$ is a general hypersurface.

Let us briefly describe the structure of this paper. In Section 2 and 3, we use the Tsen-Lang Theorem to show that $\CH_2(X_d)$ can be generated by ruled surfaces and the torsion part of $\CH_2(X_d)$ is annihilated by $d$ if $X_d$ is of $3$-Fano. In Section 4, we use the fibration structure of $\Fa(X_d)$ and the homotopy fiber sequence to prove the second homology group of $\Fa(X_d)$ is torsion-free and generated by the class of lines under some hypothesis.

 In Section 5, 6 and 7, we systematically use the techniques of rational curves such as bend-and-break, smoothing curves and the geometry of quadrics to show the connectedness of the moduli spaces of rational curves on $\Fa(X_d)$. As a result, we show Theorem \ref{mainthmab}.

\textbf{Acknowledgments.} The author would like to thank Professor~Jason Starr for his considerate explanation of his thesis to the author. The author thanks Professor~Luc~Illusie and Professor~Burt~Totaro for their interest in this project. The author also thanks Professor~Matt Kerr for giving a course on algebraic cycles in Washington University in St. Louis. One part of the paper is inspired during his course. At the end, the author is grateful for his truly great friend Dr.~Zhiyu Tian.

\section{Rationally Equivalent to Ruled Surfaces}
In this section, we suppose that $X_d$ is a hypersurface in $\PP^{n}_\mathbb{C}$ of degree d.
\begin{lemm}\label{rational}
Suppose that $E$ is a vector bundle of rank two on an algebraic scheme $M$. Let $\sigma_0$ and $\sigma_1$ be two sections of $\pi: \PP(E)\rightarrow M$. Then, we have
\[\sigma_0-\sigma_1=\pi^{*} ([N]) \in \CH^1(\PP(E))\]
where $N$ is a cycle of codimension one on $M$.
\end{lemm}
\begin{proof}
By \cite{F}, we have \begin{equation}\label{chow}
\xymatrix{ \CH_{l-1}(M)\bigoplus \CH_{l}(M) \ar[rr]^{\cong} &&\CH_l(\PP(E))}
\end{equation}
\[(Z_0,Z_1) \mapsto\pi^*(Z_0)+h\cap \pi^*(Z_1)\]
where $Z_0\in \CH_{l-1}(M)$, $Z_1\in \CH_{l}(M)$ and $h=c_1(\OO_{\PP(E)}(1))$. If $l=\dim(M)$, then the sections $[\sigma_i(M)]\in \CH^1(\PP(E))$ have the same second coordinate $h\cap \pi^*([M])$ by the identification (\ref{chow}). In other words, $[\sigma_i(M)]=(-,h\cap \pi^*([M]))$. Therefore, we have
\[\sigma_0-\sigma_1=\pi^{*} ([N])\]for some $[N]\in \CH_{l-1}(M)$.
\end{proof}
Let $\M_{0,2}(X_d,\PP^1\vee \PP^1)$ be the Kontsevich space parametrizing the reducible conics on $X_d$ with one marked point on each component. Denote $\M_{0,2}(X_d,\PP^1\vee \PP^1)$ by $\M_2$. It is clear that we have the forgetful map \[\pi:\M_2\rightarrow \M_{0,1}(X_d,\PP^1\vee \PP^1):=\M_1\]forgetting the second marked point.

There is a commutative diagram of evalution maps
\[\xymatrix{\M_2 \ar[d]^{\pi} \ar[rr]^{(ev_1,ev_2)} && X_d\times X_d\ar[d]^{\pi_1}\\
\M_1=\M_{0,1}(X_d,\PP^1\vee \PP^1)\ar@/^1pc/[u]^{\sigma_1} \ar[rr]^{ev} && X_d}.\]
where $\sigma_1$ is the universal section of $\pi$ and $\pi_1$ is the projection onto the first factor.

Let $\M_{1,p}$ be the fiber  $ev^{-1}(p)$ of $ev$ over $p\in X_d$. So we have maps
\begin{equation}\label{moduli}
\xymatrix{ C=\M_2|_{\M_{1,p}}\ar[rr]^{f=ev_2} \ar[d] &&X_d\\
\M_{1,p} \ar@/^1pc/[u]^{\sigma_2} \ar@/^3pc/[u]^{\sigma_1}}
\end{equation}
such that $\im (ev_2\circ \sigma_1) =\{p\}$ and $\sigma_2$ is the section induced by the singular locus of $\pi$. More precisely, suppose that $s\in \M_{1,p}$ parametrizes a reducible conic $L_1\cup L_2$ with the wedge point $q= L_1\cap L_2$, then $\sigma_2(s)=q$. Therefore, the fiber $f^{-1}(q)$ of $f$ over $q$ is the fiber $(ev_1,ev_2)^{-1}(p,q)$. 
In other words, the fiber $f^{-1}(q)$ parametrizes the reducible conics (on $X_d$) passing through $p$ and $q$, so $f^{-1}(q)$ is a complete intersection, see \cite{DS} or \cite[Page 82(2)]{DS2}. Moreover, the general fiber of $f$ is a complete intersection in $\PP^n$ of type
\[(1,1,2,2,3,3\ldots,d-1,d-1,d).\]
Suppose that $S$ is a surface passing through a general point of $X_d$. It follows that the preimage $f^{-1}(\Spec (\C(S))$ of the generic point of $S$ is a complete intersection in $\PP^n_{\C(S)}$ of type
\[(1,1,2,2,3,3\ldots,d-1,d-1,d).\]
In particular, if the square sum of these degrees \[2(\sum\limits_{i=1}^{d-1} i^2)+d^2=\frac{d(2d^2+1)}{3}\]is less than $n+1$, then, by the Tsen-Lang Theorem \cite{Lang}, there exists a surfaces $S'\subseteq C$ such that $f|_{S'}$ is generic one-to-one onto $S$. In other words, we have $f_*([S'])=[S]$ in $\CH_2(X_d)$. Here, we do not need to worry about the stacky issue since $S'$ passes through a non-stacky point, 
see \cite[Section 5 and 6]{V} for the details. Therefore, we have the following diagram induced by the diagram (\ref{moduli}) 
\[\xymatrix{C'\ar[r]^g \ar[d] \ar@{}[dr]|-{\Box} &C\ar[r]^f \ar[d]^{\pi}& X_d\\
S'\ar@/^1pc/[u]^{\sigma_1} \ar@/^2.5pc/[u]^{\sigma_2} \ar@/^4pc/[u]^{\Delta}  \ar[r]_{\pi|_{S'}} &\M_{1,p}}\]
where $\Delta$ is induced by the diagonal map of $S'$. It is clear that \[f\circ g\circ \Delta=f|_{S'} ~\text{and}~ \im(f\circ g\circ \sigma_1) =\{p\}.\]By Lemma \ref{rational}, the following cycles are rationally equivalent
\[[\Delta(S')] \sim_{rat} [\sigma_2(S')] \sim_{rat}[\sigma_1(S')]\]
in $\CH_2(C')$ mod $\pi^* \CH_1(S')$. In particular, \[[S]=[f\circ g\circ \Delta(S')] \sim_{rat} [f\circ g\circ \sigma_1(S)]=[p]=0\]
in $\CH_2(X_d)$ mod $\im\left((f\circ g)_*\circ \pi^*:\CH_1(S')\rightarrow \CH_2(X_d)\right)$. Therefore, we have \[[S]\in \im\left((f\circ g)_*\circ \pi^*:\CH_1(S')\rightarrow \CH_2(X_d)\right).\] We conclude that $[S]$ is the formal sum of some ruled surfaces whose fibers are lines. In particular, the Abel-Jacobi map $\AJ$ is surjective
\begin{equation}\label{AJ}
\AJ:\CH_1(\Fa(X_d))\rightarrow \CH_2(X_d)
\end{equation}
where $\Fa(X_d)$ is the Fano variety of lines of $X_d$. In summary, we show the following proposition.

\begin{propo} \label{rationaleqruledsurf}
Every surface $S$ passing through a general point in $X_d$ is rationally equivalent to the formal sum of ruled surfaces if \begin{equation}\label{inequality}
\frac{d(2d^2+1)}{3}\leq n.
\end{equation}
\end{propo}
Let $C$ be a projective curve. Suppose that $g$ is a morphism $g:C \rightarrow \Fa(X_d)$. We consider the following incidence subvariety $\II$
\begin{equation}\label{incidence}
 \xymatrix{ \II=\{(l,P)|l\subseteq P\} \ar[d]^h \ar@{^{(}->}[r] &\Fa(X_d)\times \Fa_2(X_d)\ar[dl]^{\pi_1} \ar[dr]^{\pi_2} \\
\Fa(X_d) & &\Fa_2(X_d) }
\end{equation}
where $l$ is a line in $X_d$, $P$ is a 2-plane in $X_d$ and $\Fa_2(X_d)$ is the Fano variety of 2-planes in $X_d$. We claim that the pull-back $h|_C$ of $h$ via $g$  has a section $s$.
\[\xymatrix{ g^{-1}\II  \ar@{}[dr]|-{\Box} \ar[d]_{h|_C} \ar[r] &\II \ar[d]^h \\
C\ar[r]^g  \ar@/^2pc/[u]^s & \Fa(X_d)}\]
We show the claim as follows. Recall that the space $\Fa_p$ of lines in $X_d$ throguh a point $p\in X_d$ is defined by the equations of $\PP^{n-1}$ of type
\[(1,2,3,4,\ldots,d-1,d),\]see \cite[Lemma 2.1]{Cubic}.
Therefore, the space $\Fano_q(\Fa_p)$ of lines in $\Fa_p$ through a point $q\in \Fa_p$ is defined by the equations in $\PP^{n-2}$ of type (complete intersection of equations of degree)
\begin{equation}\label{typeeq}
(1,1,2); (1,1,2,3);(1,1,2,3,4);\ldots;(1,1,2,3,\ldots d).
\end{equation}
Let $l$ be a line in $\Fa_p$ corresponding to a plane $\PP^2=\Span(l,p)$ in $X_d$. Suppose that $x$ is a point of $ \Fa(X_d)$ parametrizing a line $L_x\subseteq X_d$. The fiber $h^{-1}(x)$ is the space of 2-planes in $X_d$ containing the line $L_x$. 

Let $Q$ be a point on $L_x$. Suppose that $\Fa_Q$ is the subspace of $\Fa(X_d)$ parametrizing the lines through $Q$. Denote by $\Fano_{x}(\Fa_Q)$ the space of lines in $\Fa_Q$ parametrizing lines through $x\in \Fa_Q\subseteq \Fa(X_d)$. So the space $\Fano_{x}(\Fa_Q)$ is defined by the equations in $\PP^n$ of type as (\ref{typeeq}). Note that \[h^{-1}(x)=\Fano_{x}(\Fa_Q)\]
and the sum of (\ref{typeeq})
\begin{equation} \label{sumstar}
(d-1)+\sum \limits_{i=2}^d\frac{i(i+1)}{2}=d-2+\frac{d(d+1)(2d+1)}{12}+\frac{d(d+1)}{4}
\end{equation}
is less than $n-1$ if we assume the inequality (\ref{inequality}). So the map $h|_C$ has a section $s$ by the Tsen-Lang Theorem. We have proved the claim.

In particular, it produces a map as follows
\[
\pr_2\circ s:C\rightarrow \II\rightarrow \Fa_2(X_d).
\]Therefore, we show the following lemma:
\begin{lemm}\label{lineinplane}
Suppose that the morphism $g$ gives rise to a family $\mathcal{L}$ of lines on $X_d$. Then the morphism $\pr_2\circ s$ gives rise to a family $\mathcal{P}$ of 2-planes in $X_d$ which contains $\mathcal{L}$.
\[\xymatrix{\mathcal{L}\ar[dr] \ar@{^{(}->}[r] & \mathcal{P} \ar[r]^{H} \ar[d] & X_d\\
&C&}\]
\end{lemm}

\section{The Torsion part of the Second Chow group}
In this section, we always assume that \begin{equation}\label{ineq}
\frac{d(2d^2+1)}{3}\leq n ~\text{and}~ d\geq 3.
\end{equation}
\begin{lemm} \label{lemm31}
With the notations as in section 2, we have
\[d\left(H_*([\mathcal{L}])\right)=\sum_i \pm [P_i] + [V]\cdot c_1(\OO_{X_d}(d))\]
holds in $\CH_2(X_d)$ where $\{P_i\}$ are 2-planes in $X_d$ and $V$ is $H_*([\mathcal{P}]) \in \CH_3(X_d)$ (see the diagram at the end of Section 2 before Lemma \ref{lineinplane} for the definition of $H$).
\end{lemm}
\begin{proof}
It is clear that \[[V]\cdot c_1(\OO_{X_d}(d))= H_*\left([\mathcal{P}]\cdot H^*(c_1(\OO_{X_d}(d)))\right)\] in $\CH_3(|V|\cap |X_d|)$, see \cite{F}. By Lemma \ref{lineinplane}, for the fiber $\mathcal{P}_y$ of $\mathcal{P}/C$ over the point $y\in C$, we have that
\begin{center}
$\PP^1=\mathcal{L}_y \subseteq \mathcal{P}_y=\PP^2$ and $[H^*([X_d])]|_{\mathcal{P}_y}=d[\mathcal{L}_y]$.
\end{center}
Therefore, we conclude that
\[d\cdot [\mathcal{L}]-H^*(c_1(\OO_{X_d}(d)))=\sum \pm \textbf{fibers of } \mathcal{P}/C=\sum \limits_{i=1}^n \pm [\mathcal{P}_i]\]
where $\{\mathcal{P}_i\}$ are 2-planes in $X_d$. It follows from the projection formula that
\[d\cdot H_*([\mathcal{L}])=\sum \limits_{i=1}^n \pm H_*([\mathcal{P}_i]) +c_1(\OO_{X_d}(d))\cdot H_*([\mathcal{P}])=\sum\limits_i \pm P_i +[V]\cdot c_1(\OO_{X_d}(d)).\]
\end{proof}

By the moving lemma, every 2-cycle of $X_d$ is rationally equivalent to the formal sum of some surfaces passing through general points of $X_d$. By Proposition \ref{rationaleqruledsurf} and Lemma \ref{lemm31}, we conclude the following corollary. 
\begin{cor}\label{corre}
Under the hypothesis as above, we have \[d[S]=\sum\limits_{i=0}^n \pm [P_i]+[V]\cdot c_1(\OO_{X_d}(d))\]holds in $\CH_2(X_d)$ where $S\in \CH_2(X_d)$, $[V]\in \CH_3(X_d)$ and $\{P_i\}$ are 2-planes in $X_d$.
\end{cor}
\begin{lemm}\label{lemmplanere}
Under the hypothesis as above, all the 2-planes of $X_d$ are rationally equivalent.
\end{lemm}
\begin{proof}
Let us consider the incidence $\II$ in (\ref{incidence}). By the calculation (\ref{sumstar}), the fibers of $h: \II\rightarrow \Fa(X_d)$ is zero locus of the equations of Fano type. We claim that these fibers are smooth complete intersections for $X_d$ general, therefore, the general fibers of $h$ are Fano varieties, hence, they are rationally connected for $X_d$ general. 

In fact, the incidence $\II$ is just the flag variety $\Fl(\PP^1,\PP^2;X_d)$ of $X_d$. Therefore, we can use the classical incidence method to prove $\II$ is irreducible, smooth and of expected dimension for $X_d$ general. Namely, consider the following incidence, it is clear that $\Fl(\PP^1,\PP^2;X_d)$ is the fiber of $\pr_2^{-1}([X_d])$.
\[\xymatrix{& \mathbb{I}=\{(\PP^1\subseteq\PP^2, X_d)~|~\PP^1\subseteq\PP^2 \subseteq X_d\} \ar[dl]_{\pr_1} \ar[dr]^{\pr_2}&\\
\Fl(\PP^1,\PP^2;\PP^n)& &\PP \Ho^0(\PP^n,\OO_{\PP^n}(d))}
\]
We leave an easy exercise to the reader to finish the proof of the claim.

Note that $\Fa(X_d)$ is rationally connected for $X_d$ general (in fact it is Fano, see \cite[Chapter V, Exercise 4.7]{K}) and the general fibers of $h:\II\rightarrow \Fa(X_d)$ are rationally connected. It follows from the Graber-Harris-Starr Theorem \cite{GHS} or the Tsen-Lang Theorem that $\II$ is rationally connected for $X_d$ general.  It is clear that the projection \[\pi_2: \II \rightarrow \Fa_2(X_d)\]is surjective. It implies that $\Fa_2(X_d)$ is rationally connected for $X_d$ general. By the specialization argument, we conclude that $\Fa_2(X_d)$ is rationally chain connected for every smooth hypersurface $X_d$. In particular, all the 2-planes of $X_d$ are rationally equivalent.
\end{proof}

\begin{propo}\label{killd}
With the hypothesis as above, we have \[d\cdot [\CH_2(X_d)]_{tor}=0.\]
\end{propo}
\begin{proof}
Let $a$ be an element of $[\CH_2(X_d)]_{tor}$. By Corollary \ref{corre}, we have
\begin{equation}\label{eq1}
d\cdot a=\sum \limits_i a_i [(\PP^2)_i]+V\cdot [X_d]
\end{equation}
where $a_i \in \mathbb{Z}$, $V\in \CH_3(\PP^n)$ and $\{(\PP^2)_i\}$ are 2-planes in $X_d$. Let $j$ be the inclusion $X_d\rightarrow \PP^n$. It is clear that $j_*(d\cdot a)$ is a torsion element in $\CH_2(\PP^n)=\mathbb{Z}$. In particular, we have\[0=j_*(d\cdot a)=(\sum\limits_i a_i) [\PP^2] +c [\PP^3]\cdot [X_d]=(\sum\limits_i a_i) [\PP^2] +c\cdot d\cdot [\PP^2] \] where $V=c[\PP^3] \in \CH_3(\PP^n)$  . Therefore, it follows that \[\sum\limits_i a_i+c\cdot d=0.\]
By Lemma \ref{lemmplanere}, the equality (\ref{eq1}) becomes
\begin{equation}\label{eq2}
d\cdot a=c([\PP^3]\cdot [X_d]-d\cdot [\PP^2])
\end{equation}
where $\PP^2$ is a 2-plane in $X_d$. It is clear that, under the hypothesis (\ref{ineq}), the hypersurface $X_d$ contains a projective space $\PP^3$ by \cite[Theorem 1.6]{W}. In particular, the equality (\ref{eq2}) is \[d\cdot a=c(d[\PP^2]-d[\PP^2])=0\] in $\CH_2(\PP^3\cap X_d)=\CH_2(\PP^3)$. We have proved the proposition.
\end{proof}

\section{The Homology of Fano Variety of Lines}
\begin{prop}\label{morse} (\cite[Page 153, Theorem]{Morse}, \cite{H4})
Let $X$ be a projective variety. Suppose that $Z$ is a subvariety of $X$ and $H$ is a linear hyperplane in the ambient projective space. If $X-(Z\cup H)$ is a local complete intersection, then the homomorphism
\[\pi_i((X-Z)\cap H)\rightarrow \pi_i(X-Z)\]
is an isomorphism for all $i<dim(X)-1$.
\end{prop}

\begin{propo}\label{prophomology}
Suppose that $X_d$ is a general hypersurface of degree $d$ in $\PP^n$ and \[d(d+1)\leq n-2.\] Then, we have \[\Ho_2(\Fa(X_d),\mathbb{Z})=\mathbb{Z}\] and it is generated by the class $[l]$ of lines in $\Fa(X_d)$ with respect to the Pl\"ucker embedding.
\end{propo}

\begin{proof}
Let $\PP$ be $\PP\Ho om_{\C}(V^*,\Ho^0(\PP^1,\OO_{\PP^1}(1)))$ where $X_d\subseteq \PP(V)=\PP^n$ is defined by a homogeneous polynomial $F$ of degree $d$. So $\PP$ parameterizes $(n+1)-$tuples \[[u_0,u_1,\ldots,u_n]\] of homogeneous polynomials of degree $e$ on $\PP^1$. Let $D\subseteq \PP$ be the closed subvariety parametrizing the tuples $[u_0,\ldots,u_n]$ that have a common zero in $\PP^1$. Suppose that $e=1$. The subvariety $D$ parameterizes the tuples $[u_0,\ldots,u_n]$ that $\Span(u_0,\ldots,u_n)$ in $\Ho^0(\PP^1,\OO_{\PP^1}(1))$ is one dimensional, i.e., every pair $(u_i,u_j)$ satisfies a scalar linear relation. In particular, we have
\begin{equation}\label{dimeq}
 \dim(D)=n+1 \text{~and~} \dim(\PP)=2n+1.
\end{equation}

Let $\PP_{X_d}$ be the closed subset of $\PP$ parameterizing $[u_0,\ldots,u_n]$ such that \[F(u_0,\ldots,u_n) \equiv 0.\] Then the subvariety $\PP_{X_d}$ is defined by $d+1$ homogeneous polynomials of degree $d$ in $\PP$, i.e., it is of type as follows
\begin{equation}\label{citype}
\underbrace{(d,\ldots,d)}_{d+1}.
\end{equation}
On the other hand, it is clear that the space $\Mor_1(\PP^1,X_d)$ parametrizing the morphisms whose images are lines is $\PP_{X_d}-D$. In particular, the expected dimension of $\Mor_1(\PP^1,X_d)$ is\[\dim(\PP)-(d+1)=2n+1-(d+1)=2n-d.\] The space $\Mor_1(\PP^1,X_d)$ has the expected dimension if and only if $\PP_{X_d}$ is a complete intersection. It is clear that there is topological fiberation as follows.
\[\xymatrix{ \textbf{PGL}_2(\mathbb{C}) \ar[r] & \Mor_1(\PP^1,X_d)\ar[d] \\
&\Fa(X_d)=\Mor_1(\PP^1,X_d)/\textbf{PGL}_2(\mathbb{C})}\] In particular, it follows from \cite[Chapter V, 4.3.2]{K} that \[\dim(\Mor_1(\PP^1,X_d))=\dim(\Fa(X))+\dim \textbf{PGL}_2(\mathbb{C})=2n-d-3+3=2n-d.\]Therefore, the subscheme $\PP_{X_{d}}$ of $\PP$ is a complete intersection in $\PP$ of type (\ref{citype}).

Notice that we have (\ref{dimeq}). By using general $n+2$ hyperplanes in $\PP$ to intersect with $\PP_{X_d}-D$, we conclude that $\PP_{X_d}-D$ contains a line $L\subseteq \PP$ if $d(d+1)\leq n-2$, see \cite[Ex.4.10.5]{K}.
By Proposition \ref{morse} and the homotopy fiber sequence, we have a short exact sequence
\[\xymatrix{0 \ar[r] & \pi_2(\Mor_1(\PP^1,X_d))\ar[r] &\pi_2(\Fa(X_d))\ar[r] &\pi_1(\textbf{PGL}_2(\mathbb{C}))=\mathbb{Z}/2\mathbb{Z} \ar[r] & 0 } \]
where we use the fact $\pi_2(\textbf{PGL}_2(\mathbb{C}))=0$ and $$ \pi_1(\Mor_1(\PP^1,X_d))=\pi_1(\PP_{X_d}-D)=\pi_1(\PP_{X_d})=\{1\}.$$ Note that $\Fa(X_d)$ is simply connected. By the Hurewicz theorem, the above exact short sequence is
\[\xymatrix{0\ar[r] & \Ho_2(\PP_{X_d}-D,\mathbb{Z})=\mathbb{Z}<[L]>\ar[r] & \Ho_2(\Fa(X_d),\mathbb{Z})\ar[r] & \mathbb{Z}/2\mathbb{Z} \ar[r] &0}\]
where $[L]$ is the homology class of a line $L$ in $\PP_{X_d}-D$. From this exact sequence, it is clear that $\Ho_2(\Fa(X_d))=\mathbb{Z}$ or $\mathbb{Z}\oplus \mathbb{Z}/2\mathbb{Z}$. We exclude the second case. In fact, if $\Ho_2(\Fa(X_d))=\mathbb{Z}\oplus \mathbb{Z}/2\mathbb{Z}$, then the map $\varphi$
\[\xymatrix{\mathbb{Z}<[L]>=\Ho_2(\Mor_1(\PP^1, X_d),\mathbb{Z})\ar@/^1.5pc/[rr]^{\varphi} \ar[r] &\Ho_2(\Fa(X_2)) \ar[r]_{i_*} & \Ho_2(\PP^N,\mathbb{Z})=\mathbb{Z}}\]
is an isomorphism where $i:\Fa(X_d)\rightarrow \PP^N$ is induced by the Pl\"ucker embedding. Therefore, we have a family \[f:L\times \PP^1 \rightarrow X_d\]of lines parametrized by $L$ such that the image $f_*([L\times \PP^1])=f_*([\PP^1\times \PP^1])$ of $f$ is homologous to a 2-plane in $X_d$. It is absurd since the self-intersection of any line bundle on $\PP^1\times \PP^1$ is even and the intersection number of two general hyperplanes with a 2-plane is one. We have proved the proposition.

\end{proof}
\section{Dual Graphs and Specializations}
We recall some facts about the Kontsevich space $\M_{0,m}(X,b)$ which parametrizes stable maps of genus zero and of degree $b$ into $X$. For the details, we refer to \cite{KM} and \cite{S1}.
\subsection{Notations}
We recall the combinatorial data for a stable map, namely, its dual graph. All the graphs in our paper are finite trees, i.e., they have finitely many vertices and do not have any loops. A graph consists of the following data:
\begin{enumerate}
\item there is a non-negative integer $d_v$ for each vertex $v$, which we call the degree of the vertex, and
\item there is a list $\mathcal{L}=\{p_1,\ldots, p_k\}$ of vertices which we call the marked points. The points in the list may not be distinct, the number of the point $p$ in this list is called the multiplicity of the point.
\end{enumerate}
Let us recall how to associate to a stable map its dual graph. For a stable map of genus zero \[f:C\rightarrow X,\]the dual graph $G(f)$ of $f$ is a graph with data as follows:
\begin{itemize}
\item the vertices in the graph $G(f)$ is one-to-one corresponding to the irreducible components  of $C$, e.g. the vertex $v\in v(G)$ corresponds to the irreducible component $C_v$ of $C$,
\item an edge between the vertices for every intersection point between two componenets of $C$,
\item there is a non-negative integer $d_v$ which is equal to the degree of $f_*([C_v])$ (it may be zero if the map $f$ collapses the component),
\item a marked point on the component $C_i\cong \PP^1$ contributes a point to the list $\mathcal{L}$. 
\end{itemize}
We call a graph is a good tree if all the numbers $\{d_v\}$ are one. The stable map $f$ is a good tree if its dual graph $G(f)$ is a good tree, i.e., the map $f$ does not contract any component and its image is the union of lines. 

\begin{definition}
Let $G$ be a graph as above. A subgraph $W$ of $G$ is a graph satisfying:
\begin{itemize}
\item the vertices $v(W)\subseteq v(G)$,
\item the edges $e(W)\subseteq e(G)$,
\item the marked vertices $\mathcal{L}(W)\subseteq \mathcal{L}(G)$ counting with the multiplicity,
\item the degree $d_v$ of $v\in W$ is equal to the degree with respect to $G$.
\end{itemize}
\end{definition} 
\subsection{Deformation and Specialization}
\begin{definition} Let $G$ be a dual graph and $H$ be a non-empty subgraph. We call $G_1$ a contraction of $G$ along $H$ if
\begin{itemize} 
\item the set of the vertices of $G_1$ is the disjoint union of the vetices that are not in H and a single point $\{v_H\}$, i.e., $v(G_1)=(v(G)-v(H)) \amalg \{v_H\}$,
\item the edges $e(G_1)$ are those edges of $G_1$ connecting two vertices that are not in $H$, together with edges from $v\in S\subseteq v(G)-v(H)$ to $v_H$ where the subset $S\subseteq v(G)-v(H)$ consists of the vertices in $G$ connecting with a vertex in $H$,
\item the marked vertices $\mathcal{L}(G_1)$ consist of the marked vertices in $v(G)-v(H)$ and the vertex $v_H$ with multiplicity $|\mathcal{L}(H)|$,
\item the degree $d_{v_H}$ of $v_H$ is $\sum\limits_{v\in v(H)} d_v$ and the degrees of other vertices are their degrees in $G$.
\end{itemize}
The degree $\deg(G)$ of the graph $G$ is $\sum\limits_{v\in v(G)} d_v$. In this case, we call $G_1$ a deformation of $G$ and $G$ a specialization of $G_1$ (without mentioning $H$).\\ Suppose that $G$ is a good tree. If $G_1$ is obtained by contracting a subgraph $K$ in $G$ and the total degree of $K$ is two, then we call $G_1$ is a conic deformation of $G$. Two good trees $G_2$ and $G_3$ are said to be connected by a conic deformation if there is some dual graph $K$ that is a conic deformation of $G_2$ and $G_3$. We would like to provide a typical example to illustrate it. 

\begin{example}\label{ex}
Consider the following configurations (the underlying chain of degree one rational curve is the same and the
unique marked point lies on the $i$-th component in $\Gamma_i$):
\[\Gamma_1:\underbrace{\M_{0,2}(X,1)\times_{X}\M_{0,2}(X,1)\times_X\ldots \times_X \M_{0,1}(X,1)}_{b~factors}\]
\[\Gamma_2:\underbrace{\M_{0,1}(X,1)\times_{X}\M_{0,3}(X,1)\times_X\ldots \times_X \M_{0,1}(X,1)}_{b~factors}\]
\[\ldots \ldots\]
\[\ldots \ldots\]
\[\Gamma_b:\underbrace{\M_{0,1}(X,1)\times_{X}\M_{0,2}(X,1)\times_X\ldots \times_X \M_{0,2}(X,1)}_{b~factors}\]
where the marked point of the configuration $\Gamma_i$ is on the $i$-th component of the domain of the stable map. We can use a conic deformation to connect $\Gamma_i$ and $\Gamma_{i+1}$. We only give this deformation for $i=1$ (the general case is similar). In fact, the conic deformation connecting $\Gamma_1$ and $\Gamma_2$ is given by the dual graph of the following configuration:
\begin{equation} \label{conf}
\underbrace{\M_{0,2}(X,2)\times_{X}\M_{0,2}(X,1)\times_X\ldots \times_X \M_{0,1}(X,1)}_{b-1~factors}.
\end{equation}

\end{example}

\end{definition}
Keeping the notations as before, we have a space $\M_{0,m}(X,G)$ which parametrizes the stable maps whose dual graphs are $G$ or specializations of $G$. We have evaluation maps
\begin{equation}\label{evaluation}
ev:\M_{0,1}(X,b)\rightarrow X~\text{and}~ev_{G}:\M_{0,1}(X,G)\rightarrow X.
\end{equation}

\begin{lemm} \label{conicdeform}
Let $G_1$ and $G_2$ be two dual graphs of good trees of total degree $e$, i.e., each graph consists of $e$ vertices such that each vertex is of degree $1$. Assume that $G_1$ and $G_2$ have $k$ marked points. Then we can connect $G_1$ to $G_2$ by a finite series of conic deformations.
\end{lemm}
See Example \ref{ex}. We can connect $\Gamma_i$ and $\Gamma_j$ by finitely many configurations similar as \ref{conf}).
\begin{proof}
Let $G_0$ be the dual graph such that $|v(G_0)|=e$ and $d_v=1$ for each $v\in v(G_0)$. Suppose that $G_0$ has a central vertex $v_0$ to which every other vertex is connected and all the marked points are $v_0$. To prove the lemma, it suffices to show that all the good trees can be deformed to $G_0$ by series of conic deformations. Suppose that $G$ is a good tree and $v'$ is a vertex of $G$ with the most edges. If $v'$ is not connected to every vertex, then there is some vertex $w$ connected to both $v'$ and other vertex $u$. 

We contract the subgraph $K(\subseteq G)$ of degree two which consists of $\{w,v'\}$. This new graph $G'$ is a contraction of a graph $G''$ where $G''$ is the union of $G'$ and one point $v_1$ such that $v_1$ connects to $v_K\in v(G')$. In particular, we get a conic deformation connecting $G''$ and $G$. From $G$ to $G''$, we increase the number of the edges connected to $v'$ since $u$ is connected to $v'$ in $G''$. Similarly, for any marked point which is not $v'$, there is a conic deformation that increases the number\[\#(\{p\in \mathcal{L}~|p=v'\}), \text{ see 5.1 (2).}\]
Repeat this process. We can connect $G$ to $G_0$ by series of conic deformations. We have proved the lemma.
\end{proof}

\begin{lemm}\label{starlemm}
Let $G$ be a dual graph with $d_v=0$ or $1$ for all $v\in v(G)$. Then $G$ is a specialization of a good tree $G'$. Moreover, we have
\[\deg(G)=\sum\limits_{v\in G}d_v=\sum\limits_{v'\in G'} d_{v'}=\deg(G').\]
\end{lemm}
\begin{proof}
For a vertex $v\in v(G)$ with $d_v=0$, we can pick up some vertex $w$ which is connected to $v$. Then we contract the subgraph that consists of $\{v,w\}$. Repeat this process, we have proved the lemma.
\end{proof}

\subsection{The connectedness of evaluation fibers for the unions of lines and conics}

In the following, we assume that $\Ho^2(X,\mathbb{Z})=\mathbb{Z}$ and the evalution fibers of the evalution maps
\begin{equation} \label{hypothesis}
\M_{0,1}(X,1)\rightarrow X~\text{and}~\M_{0,1}(X,2)\rightarrow X
\end{equation}
are connected and nonempty. We will verify the assumptions for $X=\Fa(X_d)$ in section 7, see Lemma \ref{line} and Lemma \ref{conic}.
It is obvious that we have the following lemma.
\begin{lemm}\label{connected}
Let $Z$ and $Y$ be Deligne-Mumford stacks over $\mathbb{C}$. The fiber product $Z\times_S Y$ of two proper morphisms \[Z\rightarrow S~and~g:Y\rightarrow S\] over $S$ is connected if $Z$ is connected and the fibers of $g$ are connected where $S$ is a $\mathbb{C}$-scheme.
\end{lemm}
Recall the evaluation map $ev_{G}:\M_{0,1}(X,G)\rightarrow X$ (\ref{evaluation}) for a graph $G$.
\begin{lemm} \label{lemmone}
Let G be a dual graph with one marked point and $d_v=1$ or $2$ for all $v\in v(G)$. Then the fibers of $ev_{G}$ are connected. Suppose that $G$ has two conic specializations $K_1$ and $K_2$, then \[\deg(G)=\deg(K_1)=\deg(K_2)\] and $ev^{-1}_{K_i}(p)\subseteq ev^{-1}_G(p)$ (in the fiber $ev^{-1}(p)$) for $p\in X$.
\end{lemm}
\begin{proof}
Assume that $k\geq 2$ (i.e., the graph $G$ has at least two vertices). Since $d_v$ is positive (i.e., a stable map of type $G$ does not contract any component), a stable map $f$ of type $G$ is reducible and it is the union of two stable maps of types $G_1$ and $G_2$ such that the marked point is on $G_1$. In other words, we have the following diagram
\[\xymatrix{ \M_{0,2}(X,G_1)\times_X \M_{0,1}(X,G_2)\ar[dr]_{ev_1} \ar@{->>}[rr] & &\M_{0,1}(X,G)\ar[dl]^{ev_G} \\
&X &}\]
where $v(G_i)=k_i$, $\deg(G_i)=b_i$ and $k_1+k_2=k$, $b_1+b_2=b$. Let $\hat{ev}_1$ and $\hat{ev}_2$ be the two natural evaluation maps of $\M_{0,2}(X,G_1)$
\[(\hat{ev}_1,\hat{ev}_2):\M_{0,2}(X,G_1)\rightarrow X\times X.\]

Note the assumption (\ref{hypothesis}). By the induction on $k$, we conclude that the fiber \[ev_1^{-1}(p)=\hat{ev}^{-1}_1(p)\times_X \M_{0,1}(X,G_2)\] is connected from the fact that the fiber $\hat{ev}_1^{-1}(p)$ is connected and Lemma \ref{connected}. In fact, the following diagram commutes and the fibers of the morphism $F$ (it is the forgetful map forgetting the second marked point) are connected. Therefore, it implies that $\hat{ev}_1^{-1}(p)$ is connected since $ev^{-1}_{G_1}(p)$ is connected by the induction on $k_1=v(G_1)$.
\[\xymatrix{\hat{ev}_1^{-1}(p)\ar[d] \ar@{}[dr]|-{\Box} \ar[r] &\M_{0,2}(X,G_1) \ar[d]^{F} \ar[r]^<<<<{\hat{ev}_1} &X \\
ev_{G_1}^{-1}(p) \ar[r] &\M_{0,1}(X,G_1)\ar[ur]_{ev_{G_1}}}
\]

\end{proof}

\begin{prop} \label{propconnectdualgraph}
Let $\Gamma_1$ and $\Gamma_2$ be two possible dual graphs of the unions of lines in $X$ of total degree $b$ with one marked point. Then the evaluation fibers 
\[ev_{\Gamma_1}^{-1}(p)~and~ev_{\Gamma_2}^{-1}(p)\]are in a common connected component of $ev^{-1}_b(p)$ where $ev_b$ is the evaluation map
\[ev_b:\M_{0,1}(X,b)\rightarrow X.\]
\end{prop}

\begin{proof}
By Lemma \ref{starlemm}, we know $\Gamma_i$ is a specialization of a good tree $\hat{\Gamma}_i$ of degree $b$, in particular, we have \[\M_{0,1}(X,\Gamma_i)\subseteq \M_{0,1}(X,\hat{\Gamma_i}).\] Therefore, we can assume that $\Gamma_i$ is a good tree for $i=1,2$. By Lemma \ref{conicdeform}, we have the following series of conic deformations to connect $\Gamma_1$ and $\Gamma_2$
\[\xymatrix{ &K_0 & &K_1&&\ldots& &K_{k+1} & \\
\hat{\Gamma}_0 \ar[ur]& \ar@{}[u]|-{conic}&\hat{\Gamma}_1\ar[ul]\ar[ur] &\ar@{}[u]|-{conic} & \hat{\Gamma}_2\ar[ul] \ar[ur] &\ldots &\hat{\Gamma_k}\ar[ur]\ar[ul] &\ar@{}[u]|-{conic} &\hat{\Gamma}_{k+1}\ar[ul] }\] where $\hat{\Gamma}_0=\Gamma_1$ and $\hat{\Gamma}_{k+1}=\Gamma_2$. Applying Lemma \ref{lemmone}, we know the fibers $ev_{\hat{\Gamma_i}}^{-1}(p)$ and  $ev_{\hat{\Gamma}_{i+1}}^{-1}(p)$ are in a common connected component of $ev_b^{-1}(p)$. We have proved the proposition.
\end{proof}
We provide an example to illustrate this proposition.
\begin{example}
With the notations as in Example \ref{ex}, we have the evaluation maps as follows
\[Ev_{\Gamma_i}:\M_{0,1}(X,\Gamma_i) \rightarrow X.\]
We use a conic deformation to connect $Ev_{\Gamma_i}^{-1}(p)$ and $Ev_{\Gamma_{i+1}}^{-1}(p)$ in $ev^{-1}_b(p)$ to show that the general fiber $Ev_{\Gamma_i}^{-1}(p)$ over a general point $p\in X$ is contained in a connected component of $ev^{-1}_b(p)$. We explain it for $i=1$, for arbitrary $i$, it is similar. By the hypothesis (\ref{hypothesis}) and Lemma \ref{connected}, we can prove the evaluation map
\[Ev:\underbrace{\M_{0,2}(X,2)\times_{X}\M_{0,2}(X,1)\times_X\ldots \times_X \M_{0,1}(X,1)}_{b-1~factors}\rightarrow X\] has connected fibers as above, where the first factor parametrizes the conics with two marked points. It is clear that the fiber $Ev^{-1}(p)$ over the point $p\in X$ contains fibers $Ev_{\Gamma_1}^{-1}(p)$  and $Ev_{\Gamma_2}^{-1}(p)$. Hence, we use this conic deformation to connect \[Ev_{\Gamma_1}^{-1}(p)~ and~Ev_{\Gamma_{2}}^{-1}(p).\]

\end{example}

\section{Special Loci, Geometry of Spaces of Quadrics}
In this section, we always assume, unless otherwise noted, that \begin{equation}\label{sixstar}
n\geq \binom {d+1}2 +d~\text{and}~ d\geq 3.
\end{equation}

\begin{prop}\label{propirr}
The stack $\M_{0,0}(\Fa(X_d),1)$ is smooth and irreducible for $X_d$ general. Moreover, we have \[\dim(\M_{0,0}(\Fa(X_d),1))=3n-\binom {d+1}2 -d-5\]for $X_d$ general.
\end{prop}
\begin{proof}
It is easy to prove this proposition by considering the smooth incidence subvariety $I$ as follows:
\[ \xymatrix{ & I=\{(l,X_d)| l\subseteq \Fa(X_d)\}\ar[ld]_{pr_1} \ar[rd]^{pr_2} &\\
\M_{0,0}(G(2,n+1),1) & &\PP(\Ho^0(\PP^n,\OO_{\PP^n}(d)))}
\]where $l$ is a line in $G(2,n+1)$. The line $l$ sweep out a 2-plane $\PP^2$ in $\PP^n$, the fiber of $pr_1$ over $[l]$ is the projective space parametrizing hypersurfaces of degree $d$ containing this 2-plane. Therefore, we conclude $I$ is smooth since $\M_{0,0}(G(2,n+1),1)$ is smooth, see \cite{FP}. On the other hand, we can follow the same method as the proof of \cite[Theorem 4.3]{K} to show the codimension of the singular locus of $pr_2$ is at least $2$. Hence, it is the same as \cite[Theorem 4.3]{K} that we can conclude $\M_{0,0}(\Fa(X_d),1)$ is smooth and connected of expected dimension for $X_d$ general. We leave the details to the reader.
\end{proof}

\subsection{Special Loci} We have a natural map
\[\M_{0,0}(\Fa(X_d),1)\times \M_{0,0}(\PP^1,2) \rightarrow \M_{0,0}(\Fa(X_d),2)\]
\[(f,g)\mapsto f\circ g\]whose image parametrizes double lines in $\Fa(X_d)$. Denote this image by $L_1$. We call the locus $L_1$ is special. Similarly, we have a closed substack $L$ of $\M_{0,0}(G(2,n+1),2) $ parametrizing double lines in $G(2,n+1)$. It is clear that the dimension of the special locus $L_1$ is
\[ \dim(\M_{0,0}(\Fa(X_d),1) )+ \dim (\M_{0,0}(\PP^1,2))=3n-\binom {d+1}2-d.\]
Moreover, it follows from Proposition \ref{propirr} that the special locus $L_1$ is irreducible for $X_d$ general. Similarly, we know $L$ is also irreducible since $\M_{0,0}(G(2,n+1),1)$ is irreducible, see \cite{Kim}.

We have the following incidence correspondence:
\begin{equation} \label{picstar}
\xymatrix{ & I=\{(C,X_d)| C\subseteq \Fa(X_d)\}\ar[ld]_{pr_1} \ar[rd]^{pr_2} &\\
\M_{0,0}(G(2,n+1),2) & &\PP(\Ho^0(\PP^n,\OO_{\PP^n}(d)))}
\end{equation} 
where $C$ is a conic in $G(2,n+1)$ and $\PP(\Ho^0(\PP^n,\OO_{\PP^n}(d)))$ is the space parametrizing hypersurfaces of degree $d$ in $\PP^n$. It is clear that $pr_2$ is surjective if any general hypersurface of degree $d$ contains a 2-plane.

\subsection{Geometry of Conics}
\begin{lemm} \label{lemmnotirr}
The special locus $L_1$ is not an irreducible component of \[\M_{0,0}(\Fa(X_d),2)\] for $X_d$ general.
\end{lemm}
\begin{proof}
We know $L_1$ is irreducible by Proposition \ref{propirr}. It is clear that a general point of $L_1$ parametrizes a conic which is a double cover of a free line passing through a general point of $\Fa(X_d)$. Recall that the space $Y$ of lines in $X_d$ through a point $p$ is defined by the equations in $\PP^n$ of type \begin{equation}\label{degrees}
(1,1,2,3,\dots,d-1,d).
\end{equation}For each point $y\in Y$, the space $Y$ has a line through $y$ if the sum of these degrees (\ref{degrees}) is at most $n-1$, see \cite[Chapter V Exercise 4.10.5 ]{K}. It follows that $\Fa(X_d)$ is covered by lines. Therefore, for a general line $\PP^1$ in $\Fa(X_d)$, it is free and we have
\[T_{\Fa(X_d)}|_{\PP^1}=\OO_{\PP^1}(a_1)\oplus\OO_{\PP^1}(a_2)\oplus \ldots \oplus \OO_{\PP^1}(a_{2n-d-3})\]where $T_{\Fa(X_d)}$ is the tangent sheaf of $\Fa(X_d)$, $a_i\geq 0$ and $a_i\geq a_{i+1}$. 

We claim that $a_3\geq 1$. In fact, if $a_3=a_4=\ldots=a_{2n-d-3}=0$, then we have \[a_1+a_2=\sum\limits_{i=1}^{2n-d-3} a_i =\deg( T_{\Fa(X_d)})=n+1-\binom {d+1}2\geq 4.\] The space $\Mor_1(\PP^1, \Fa(X_d); \{0,1\})$ of maps fixing $0$ and $1$ is the disjoint union of copies of $\mathbb{C}^*$. On the other hand, the dimenison of the tangent space of $$\Mor_1(\PP^1, \Fa(X_d); \{0,1\})$$ is 
$h^0(\PP^1,T_{\Fa(X_d)}(-2)|_{\PP^1})=h^0(\PP^1, \OO_{\PP^1}(a_1-2)\oplus \ldots \oplus \OO_{\PP^1}(a_{2n-d-3}-2))\geq 2$ which is a contradiction. So we prove our claim. 

In particular, by \cite[Theorem 3.14.3, Chapter II]{K}, a general deformation of a double cover of a general (free) line is a smooth conic in $\Fa(X_d)$. It implies that $L_1$ is not an irreducible component of $\M_{0,0}(\Fa(X_d),2)$.
\end{proof}

We apply a similar method to prove the following lemma.

\begin{lemm} \label{lemmairr}
The incidence correspondence $I$ (\ref{picstar}) is irreducible.
\end{lemm}
\begin{proof}
Denote by $U$ the open substack $\M_{0,0}(G(2,n+1),2)-L$. By the main theorem of \cite{Kim}, we know $\M_{0,0}(G(2,n+1),2)$ is irreducible. Hence, the stack $U$ is irreducible. On the other hand, every point $c\in U$ sweeps out a quadric $Q_c\in \PP^n$ (it may be singular), the fiber of $pr_1$ over $c$ is the projective space  parametrizing hypersurfaces containing $Q_c$. The projective space is a subspace of $\PP \Ho^0(\PP^n,\OO_{\PP^n}(d))$ of codimension $(d+1)^2$, in fact, it is given by
\begin{equation} \label{quadric}
\xymatrix{\PP(\mathrm{Ker}:\Ho^0(\PP^n,\OO_{\PP^n}(d))\ar@{->>}[r] &\Ho^0(Q_c,\OO_{Q_c}(d)))}.
\end{equation} In particular, the open substack $pr_1^{-1}(U)$ of $I$ is irreducible. On the other hand, a point $u\in L$ parametrizing a double line in $\Fa(X_d)$ sweeps out a 2-plane in $\PP^n$. As above, the fiber of $pr_1$ over $u$ is the projective space as follows:
\[\xymatrix{\PP(\mathrm{Ker}:\Ho^0(\PP^n,\OO_{\PP^n}(d))\ar@{->>}[r] &\Ho^0(\PP^2,\OO_{\PP^2}(d)))}.\] In particular, the closed substack $pr_1^{-1}(L)$ is irreducible.

We claim that $pr_1^{-1}(L)$ is not an irreducible component of $I$. The proof is similar to the proof of Lemma \ref{lemmnotirr}. Since a general hypersurface of degree $d$ in $\PP^n$ contains a 2-plane, a general point $u$ in $pr^{-1}_1(L)$ maps to a general point of $\PP \Ho^0(\PP^n,\OO_{\PP^n}(d))$ via $pr_2$. Denote $pr_2(u)$ by $[X_d]$. On the other hand, $L_1= pr_1^{-1}(L)\cap pr_2^{-1}([X_d])$. By the argument as in the proof of Lemma \ref{lemmnotirr}, we can deform the double line in $\Fa(X_d)$ associated to $u$ to a smooth conic in $\Fa(X_d)$. In particular, the preimage $pr_1^{-1}(L)$ is not an irreducible component of $I$.

In summary, we prove that $I$ is irreducible.

\end{proof}

\begin{lemm} \label{lemmsm}
The space $\M_{0,0}(\Fa(X_d),2)-L_1$ is smooth (as a scheme) for $X_d$ general.
\end{lemm}
\begin{proof}
We show the stack $I-pr_1^{-1}(L)$ is smooth first. In fact, the stack $I-pr_1^{-1}(L)$ is the preimage of the smooth stack $\M_{0,0}(G(2,n+1),2)-L$ (see \cite{FP}) via $pr_1$. The fiber of $pr_1$ over a point $[C]\in \M_{0,0}(G(2,n+1),2)-L$ is just the projective subspace parametrizing hypersurfaces of degree $d$ containing the quardic $Q_C\subseteq \PP^n$ where $Q_C$ is swept out by the lines parametrized by $C$, see (\ref{quadric}).
It implies that the stack $I-pr_1^{-1}(L)$ is smooth. It is clear that \[\M_{0,0}(\Fa(X_d),2)-L_1=(I-pr_1^{-1}(L))\cap pr_2^{-1}([X_d]).\] Since the space $I-pr_1^{-1}(L)$ is smooth. We conclude this lemma by the generic smoothness theorem via $pr_2$.
\end{proof}

\subsection{Geometry of Quadrics}
\begin{lemm}\label{lemmcon}
The space $\M_{0,0}(\Fa(X_d),2)$ is connected if \[n \geq \binom {d+1}2+d-1~\text{and}~d\geq 3.\]
\end{lemm}
\begin{proof}
We use degeneration argument. We can assume that $X_d$ is general. By deformation theory, we know every irreducible component of $\M_{0,0}(\Fa(X_d),2)$ is of dimension at least (expected dimension=)
\[\begin{aligned}
&-K_{\Fa(X_d)}\cdot 2[line]+\dim(\Fa(X_d))-3\\
&=2 (n+1-\binom {d+1}2)+2n-3-d-3\\
&\geq 4n-\binom{d+1}2-d-5
\end{aligned}\]
where 
 \[-K_{\Fa(X_d)}=\OO_{\Fa(X_d)}(n+1-\binom {d+1}2)\] and $\dim(\Fa(X_d))=2n-3-d$ (see \cite[Chapter V Theorem 4.3 and Exercise 4.7]{K}). By the assumption, we conclude that every irreducible component of $\M_{0,0}(\Fa(X_d),2)$ has dimension at least $3n-6$. We apply \cite[Proposition 11.6]{Eric} to conclude that every component contains a point parametrizing a conic $C$ in $\Fa(X_d)$ such that the associated quadric $Q_C$ in $\PP^n$ is singular. Therefore, a such conic $C$ is in a subspace $Y$ of $\Fa(X_d)$ parametrizing the lines through a point. Recall that $Y$ is defined by the equations in $\PP^n$ of type
\[(1,1,2,3,\ldots,d-1,d).\] If\[1+1+2+3+\ldots+(d-1)+d +\frac{d+1}{2}\leq \frac{3n-2}{2}\] (i.e., $\frac{(d+1)^2+4}{3}\leq n$), then the subspace $\M_{0,0}(Y,2)$ of $\M_{0,0}(\Fa(X_d),2)$ is connected by \cite[Theorem 1.3]{Zongrunpu} and intersects the special locus $L_1$ by the fact that $Y$ contains a line, see \cite[Chapter V, Exercise 4.10.5]{K}. It implies the connectedness of $\M_{0,0}(\Fa(X_d),2)$.

\end{proof}

\begin{lemm} \label{lemmlc}
Every irreducible component of space $\M_{0,0}(\Fa(X_d),2)$ is of dimension \[D=\dim~ \M_{0,0}(G(2,n+1),2)-(d+1)^2.\] The space $\M_{0,0}(\Fa(X_d),2)$ is a local complete intersection in $\M_{0,0}(G(2,n+1),2)$, therefore, it is $(D-1)$-connected, see \cite[Chapter 3 3.3.C ]{L}.
\end{lemm}

\begin{proof}
We claim $\M_{0,0}(\Fa(X_d),2)$ is the zero locus of a section of a vector bundle on $\M_{0,0}(G(2,n+1),2)$. In fact, we have the following universal bundles over $G(2,n+1)=G$ and $\M_{0,0}(G(2,n+1),2)$ 
\[\xymatrix{ \mathcal{L}=\PP(S)\ar[d]^{\pi}\ar[r]^f &\PP^n & \mathcal{C}\ar[r]^g\ar[d]^{\hat{\pi}}& G(2,n+1)\\
G(2,n+1) & & \M_{0,0}(G,2) & &}\]where $S$ is the universal subbundle over $G(2,n+1)$. Suppose that the section $s\in \Ho^0(\PP^n,\OO_{\PP^n}(d))$ defines the hypersurface $X_d\subseteq \PP^n$. We have 
\[\pi_*f^*\OO_{\PP^n}(d)=\pi_*\OO_{\PP(S)}(d)=\Sym^d(S^{\vee})\]
and a section $s_1=\pi_*f^*(s)\in \Ho^0(G(2,n+1),\Sym^d(S^{\vee}))$. The zero locus of $s_1$ defines the Fano variety $\Fa(X_d)$ of lines in $X_d$. \\
We claim that \[\hat{\pi}_*g^*\Sym^d(S^{\vee})=E\]is a vector bundle on $\M_{0,0}(G(2,n+1),2)$. Moreover, it has a section $s_2=\hat{\pi}_*g^*(s_1)$. In fact, by the base change theorem, it suffices to prove that the cohomology group 
\begin{center}
$\Ho^1(C,h^*(\Sym^dS^{\vee}))$ is zero 
\end{center} for any conic $h:C\rightarrow G=G(2,n+1)$. We divide it into two cases:
\begin{itemize}
\item 
If $h:\PP^1 \rightarrow G$ is a smooth conic or a double line, then $h^*(\Sym^d(S^{\vee}))$ is semi-ample since there is quiotient $\OO_G^{\oplus {n+1}} \rightarrow  S^{\vee}.$ In particular, we have $\Ho^1(C,h^*(\Sym^d(S^{\vee}))=0$.
\item
If the image $h:C=\PP^1\vee \PP^1\rightarrow G$ is a reducible conic or a double line, then, as above, we conclude $\Ho^1(C,h^*(\Sym^dS^{\vee}))=0$ by using the short exact sequence 
\[0\rightarrow \OO_{\PP^1}(-1)\rightarrow \OO_C\rightarrow \OO_{\PP^1}\rightarrow 0.\]
\end{itemize}

It is clear that the zero locus of $s_2$ defines $\M_{0,0}(\Fa(X_d),2)$.
 As in the proof of Lemma \ref{lemmsm}, we know that \[pr_2^{-1}([X_d])=\M_{0,0}(\Fa(X_d),2),\]cf. (\ref{picstar}). Note that $I$ is irreducible by Lemma \ref{lemmairr} and $L_1$ is not an irreducible component of $\M_{0,0}(\Fa(X_d),2)$ for $X_d$ general by Lemma \ref{lemmnotirr}. The dimension of each irreducible component of $\M_{0,0}(\Fa(X_d),2)$ for $X_d$ general is \[\begin{aligned}
  &\dim(I)-\dim( \PP \Ho^0(\PP^n,\OO_{\PP^n}(d)))\\
  &=\dim (\M_{0,0}(G(2,n+1),2))+ \dim (pr_1^{-1}(C))-\dim( \PP\Ho^0(\PP^n,\OO_{\PP^n}(d)))\\
&=\dim (\M_{0,0}(G(2,n+1),2))-h^0(Q_C,\OO_{Q_C}(d))\\
&=\dim (\M_{0,0}(G(2,n+1),2))-(d+1)^2
\end{aligned} \]
where $[C]$ is a general point of $\M_{0,0}(G(2,n+1),2)$. Hence, the fiber $pr_1^{-1}([C])$, cf. (\ref{picstar}), is the following projective space by (\ref{quadric})
\[\PP\left(\mathrm{Ker}:\Ho^0(\PP^n,\OO_{\PP^n}(d))\rightarrow \Ho^0(Q_C,\OO_{Q_C}(d))\right).\]
We show the first statment of the lemma. 

To prove the space $\M_{0,0}(\Fa(X_d),2)$ is a local complete intersection in \[\M_{0,0}(G(2,n+1),2),\] it suffices to verify the rank of $E$ is $(d+1)^2$ since $\M_{0,0}(\Fa(X_d),2)$ is the zero locus of a section of $E$. It is clear that the rank of $E$ is $h^0(C,q^*(\Sym^dS^{\vee}))$ for any conic $q:C\rightarrow G$. So we can consider the case when $q:C\rightarrow G$ is a double line. Since $\varphi^*(S^{\vee})=\OO_{\PP^1}(1)\oplus \OO_{\PP^1}$ for any line $\varphi:\PP^1 \rightarrow G$, we conclude that \[\begin{aligned}
&\varphi^*\Sym^dS^{\vee}\\
&=\Sym^d(\OO_{\PP^1}(1)\oplus \OO_{\PP^1})\\
&=\sum\limits^{d}_{s=0} \OO_{\PP^1}^{\otimes s}\otimes \OO_{\PP^1}(1)^{\otimes (d-s)}=\sum\limits^{d}_{s=0} \OO_{\PP^1}(s).\end{aligned} \] For the double line $q$ induced by $\varphi$, i.e., 
\[\xymatrix{q:\PP^1\ar[r]^{\psi}_{2:1}& \PP^1 \ar[r]^{\varphi} &G},\] we have \[\begin{aligned}
& h^0(\PP^1,q^*\Sym^dS^{\vee})\\
&=h^0(\PP^1,\psi^*(\varphi^*\Sym^dS^{\vee}))\\
&=\sum\limits^d_{s=0}h^0(\PP^1,\OO_{\PP^1}(2s))=\sum\limits_{s=0}^d(2s+1)=(d+1)^2.
\end{aligned}\] We have proved the lemma.
\end{proof}

\begin{cor}\label{corconnect}
The space $\M_{0,0}(\Fa(X_d),2)-L_1$ is connected.
\end{cor}
\begin{proof}
It follows from Lemma \ref{lemmlc} and the fact that the codimension of $L_1$ in \[\M_{0,0}(\Fa(X_d),2)\] is at least 2.
\end{proof}
\begin{prop}\label{propconnirr}
The space \[\M_{0,m}(\Fa(X_d),2)\] is irreducible for $X_d$ general.
\end{prop}

\begin{proof}
Note that the generic fibers of \[\M_{0,m}(\Fa(X_d),2)\rightarrow \M_{0,m-1}(\Fa(X_d),2)\] are irreducible curves. It suffices to prove the case $m=0$ by the induction on $m$. It follows from Lemma \ref{lemmsm} and Corollary \ref{corconnect} that the space $\M_{0,0}(\Fa(X_d),2)-L_1$ is irreducible for $X_d$ general. Since $L_1$ is not an irreducible component of $\M_{0,0}(\Fa(X_d),2)$ for $X_d$ general by Lemma \ref{lemmnotirr}, we conclude that $\M_{0,0}(\Fa(X_d),2)$ is irreducible for $X_d$ general.
\end{proof}

\section{Bend and Break, Curves on Fano Variety of Lines}

In this section, we assume that $X_d$ is a smooth hypersurface of degree $d$ in $\PP^n$.
\begin{theorem}\label{propgriffiths}Suppose that one of the following assumptions holds:
\begin{itemize}
\item  $d\geq 4$ and \[n\geq 3\binom {d+1}{2}-d-4,\]
\item $d=3$ and $n\geq 14$.
\end{itemize}If the Fano variety $\Fa(X_d)$ is smooth, then we have
\[\CH_1(\Fa(X_d))_{hom}=\CH_1(\Fa(X_d))_{alg}\] i.e., $\mathrm{Griff_1}(\Fa(X_d))=0$. In particular, $\CH_1(\Fa(X_d))_{hom}$ is divisible.
\end{theorem}
\begin{proof}
Note that $\Fa(X_d)$ is Fano under the assumption of the theorem. It follows from \cite[Theorem 1.3]{TZ} that the first Chow group $\CH_1(\Fa(X_d))$ is generated by rational curves. Suppose that 1-cycle  $C$ is a rational curve. It is homologus to $b\cdot[line]$, see Proposition \ref{prophomology}. The cycle $C$ is algebraically equivalent to the sum of some lines in $\Fa(X_d)$ since $\M_{0,0}(\Fa(X_d),b)$ is connected by Proposition \ref{conn}. We have proved the theorem.
\end{proof}

In the rest of this section, we will show Proposition \ref{conn}. In the following, we assume that the hypersurface $X_d$ is general and one of the assumptions of Theorem \ref{propgriffiths} holds.

\begin{lemm}\label{line}
The fibers of the evalution map \[ev:\M_{0,1}(\Fa(X_d),1)\rightarrow \Fa(X_d)\] are connected and nonempty.
\end{lemm}
\begin{proof}
The fiber $ev^{-1}([l])$ over $[l]\in \Fa(X_d)$ parametrizes 2-planes in $X_d$ containing the line $l$. By the same arugment as in section 2 (after Proposition \ref{rationaleqruledsurf}), we know the fiber $ev^{-1}([l])$ is the intersection of ample divisors in a projective spaces and of positive dimension, therefore, it is connected and nonempty. 
\end{proof}

\begin{lemm} \label{connectedfibers}
Suppose that the map $f:X\rightarrow Y$ is proper from an irreducible reduced Deligne-Mumford stack $X$ to an irreducible variety $Y$ over complex numbers $\mathbb{C}$. Let $Z$ be a closed Deligne-Mumford substack of $X$. If the fibers of $f|_Z:Z\rightarrow Y$ are nonempty and connected, then the fibers of $f$ are connected.
\end{lemm}

\begin{proof}
According to \cite{Martin}, we have the Stein factorization as follows
\[\xymatrix{ X\ar@/^1pc/[rr]^f \ar[r]_g& W \ar[r]_h & Y}\]where $W$ is an integral scheme, the morphism $h$ is finite and and the fibers of $g$ are connected. Since $X$ is irreducible, the preimage $X_U$ of an open subset $U$ of $Y$ is irreducible (in particular, it is connected). So we can shrink $Y$ to an open subset $U$ so that $h$ over $U$ is \'etale. The fibers $h^{-1}(u)=\{u_1,\ldots, u_s\}$ of $h$ over $u\in U$ correspond to the connected components of $f^{-1}(u)$. There is a unique marked point $P$ in $h^{-1}(u)$ associated to the connected component of $f^{-1}(u)$ containing $Z_u$. Since $X_U$ is connected, the fundamental group $\pi_1(U)$ acts on the fibers $h^{-1}(u)$ transitively and fixes the marked point $P$. It implies that the morphism $h$ is one-to-one. In particular, the general fibers of $f$ are connected. Therefore, all the fibers of $f$ are connected.
\end{proof}

\begin{lemm}\label{conic}
The fibers of the evalution map \[ev:\M_{0,1}(\Fa(X_d),2)\rightarrow \Fa(X_d)\] are connected and nonempty.
\end{lemm}
\begin{proof}
Recall that $\Fa(X_d)$ is covered by lines, see Lemma \ref{line} or the proof of Lemma \ref{lemmnotirr}. For any point $u\in \Fa(X_d)$, there is a line $l\subseteq \Fa(X_d)$ passing through $u$. A double cover from $\PP^1$ to $l$ gives rise to a point $[f]\in \M_{0,1}(\Fa(X_d),2)$ such that $ev([f])=u$. In particular, the fibers of $ev$ are non-empty. By Proposition \ref{propconnirr}, we can apply Lemma \ref{connectedfibers} to $X=\M_{0,1}(\Fa(X_d),2)$, $Y=\Fa(X_d)$ and $f=ev$. We take \[Z=\im\left( \M_{0,2}(\Fa(X_d),1)\times_{\Fa(X_d)} \M_{0,1}(\Fa(X_d),1)\rightarrow \M_{0,1}(\Fa(X_d),2)\right). \]It follows from Lemma \ref{connected} and Lemma \ref{line} that the fibers of $ev|_Z:Z\rightarrow \Fa(X_d)$ are connected. We have proved the lemma.
\end{proof}

\begin{propo}\label{conn}
With the same assumption as Theorem \ref{propgriffiths}, for a general hypersurface $X_d$, the fibers of the evaluation map \[ev:\M_{0,1}(\Fa(X_d),b)\rightarrow \Fa(X_d)\] are connected for any $b\in \mathbb{N}_{+}$ and intersect the locus of the unions of lines. In particular, the space $\M_{0,0}(\Fa(X_d),b)$ is connected for any smooth $\Fa(X_d)$.
\end{propo}

\begin{proof}
Denote the curve class \[b[l]\in \Ho_2(\Fa(X_2),\mathbb{Z})=\mathbb{Z}<[l]>\]by $b$, see Proposition \ref{prophomology}. Since the map \[ev_1:\M_{0,1}(\Fa(X_d),1)\rightarrow \Fa(X_d)\] is surjective by Lemma \ref{line}, there are unions of lines passing through $p$ for every point $p\in \Fa(X_d)$. In particular, if the evaluation fibers of $ev$ are connected, then the fibers intersect the locus of unions of lines 

To show the proposition, it suffices to either prove 
\begin{itemize}
\item the fiber of $ev$ is connected or prove
\item every connected component of the evaluation fiber intersects the locus of unions of lines and the locus of unions of lines in the evaluation fiber is in a common connected component of the evaluation fiber.
\end{itemize}
We show the statement by the induction on $b$. For $b=1$ and $2$, it is Lemma \ref{line} and Lemma \ref{conic}. Let us assume $b\geq 3$. Suppose that $f$ is a map \[f:\PP^1\rightarrow \Fa(X_d)~\text{and}~f_*([\PP^1])=b=b[l]\in \Ho_2(\Fa(X_d),\mathbb{Z}).\] Assume that $f(0)=p$ and $f(\infty)=q$, we consider the morphism space\[\Mor(\PP^1,\Fa(X_d);0\mapsto p,\infty\mapsto q).\] By \cite[Chapter II]{K}, its expected dimension at the point $[f]$ is
\[\begin{aligned}
&h^0(\PP^1,T_{\Fa(X_d)}(-2))-h^1(\PP^1,T_{\Fa(X_d)}(-2))\\
&=-K_{\Fa(X_d)}\cdot [f(\PP^1)]-\dim(\Fa(X_d))\\
&=b~(n+1-\binom {d+1}2)-(2n-3-d)\geq 2
\end{aligned} \]
where the last equality follows from \cite[Chapter V Theorem 4.3 and Exercise 4.7]{K}. In particular, by the bend-and-break \cite[Chapter II]{K}, every irreducible component of ev intersects the boundary $B\subseteq \M_{0,1}(\Fa(X_d),b)$ for $b\geq 3$. Denote $\Fa(X_d)$ by $Y$. Recall that there are natural gluing maps for $(b_1,b_2)$ \[\M_{0,2}(Y,b_1)\times_{Y}\M_{0,1}(Y,b_2)\rightarrow \M_{0,1}(Y,b).\] and \[B=\bigcup\limits_{b_1+b_2=b,b_1,b_2>0} \mathrm{Im}\left(\M_{0,2}(Y,b_1)\times_Y\M_{0,1}(Y,b_2)\rightarrow \M_{0,1}(Y,b)\right).\]
We claim that the fibers of \[ev_1:\M_{0,2}(Y,b_1)\times_{Y}\M_{0,1}(Y,b_2) \rightarrow Y\] are connected and intersect the locus of unions of lines. In fact, as at the beginning of the proof, for every point $p\in Y$, there are unions of lines passing through $p$. In particular, we only need to prove the fibers of $ev_1$ are connected.  Note that we have the following diagram
\begin{equation} \label{ev1}
\xymatrix{\hat{ev}_1^{-1}(p)\ar[d] \ar@{}[dr]|-{\Box} \ar[r] &\M_{0,2}(Y,b_1) \ar[d]^{F} \ar[r]^<<<<{\hat{ev}_1} &Y \\
ev^{-1}_{b_1}(p) \ar[r] &\M_{0,1}(Y,b_1)\ar[ur]_{ev_{b_1}}}
\end{equation}
where $F$ is the forgetful functor forgetting the second marked point.
By the induction on $b$, we know $ev^{-1}_{b_1}(p) $ is connected, therefore, the evaluation fiber $\hat{ev}_1^{-1}(p)$ in the diagram (\ref{ev1}) is connected. Since the fiber of $ev_1$ over $p$ is \[ev_1^{-1}(p)=\hat{ev}_1^{-1}(p)\times_Y \M_{0,1}(Y,b_2),\] we show the claim by Lemma \ref{connected}.

In summary, we conclude that every connected component of the fiber $ev^{-1}_b(p)$ parameterizes some stable maps with one marked point whose images are unions of lines. These stable maps are in a common connected component of $ev_b^{-1}(p)$ by Proposition \ref{propconnectdualgraph} ($X=Y$). We have proved the proposition.

\end{proof}

\begin{theorem}
Suppose that $X_d$ is a hypersurface of degree $d$ in $\PP^n$ with smooth Fano variety of lines. If the equality \[\frac{d(2d^2+1)}{3}\leq n\]holds, then $\CH_2(X_d)=\mathbb{Z}$.
\end{theorem}

\begin{proof}
By the main theorem of paper \cite{E}, it suffices to prove that \[\CH_2(X_d)_{tor}=0.\] We claim that $\CH_2(X_d)_{tor}$ is divisable. In fact, we have that \[\CH_2(X_d)_{tor}\subseteq \CH_2(X_d)_{hom}\]
where $\CH_2(X_d)_{hom}$ is the subgroup of $\CH_2(X_d)$ generated by 2-cycles that are homologous to zero.
To prove $\CH_2(X_d)_{tor}$ is divisible, it suffices to show $\CH_2(X_d)_{hom}$ is divisible. Let us consider the following commutative diagram.
\[\xymatrix{ 0\ar[r]& \CH_1(\Fa(X_d))_{hom}\ar[r]\ar[d]^{\AJ_{hom}} &\CH_1(\Fa(X_d))\ar[r]^{cl} \ar[d]^{\AJ} & \Ho_2(\Fa(X_d),\mathbb{Z})\ar[r]\ar[d]^{\AJ_{top}} &0\\
0\ar[r]& \CH_2(X_d)_{hom}\ar[r] &\CH_2(X_d)\ar[r]^{cl} & \Ho_4(X_d,\mathbb{Z})\ar[r] &0 }\]
By Proposition \ref{rationaleqruledsurf}, the map $\AJ$ (\ref{AJ}) is surjective. By Proposition \ref{prophomology}, the map $\AJ_{top}$ is an ismorphism. Therefore, the map $\AJ_{hom}$ is surjective. Note that $\CH_1(\Fa(X_d))_{alg}$ is divisible. It implies that $\CH_2(X_d)_{hom}$ is divisible since $\CH_1(\Fa(X_d))_{hom}=\CH_1(\Fa(X_d))_{alg}$ by Theorem \ref{propgriffiths}.

It follows from Proposition \ref{killd} and the claim $\CH_2(X_d)_{tor}$ is divisible that \[\CH_2(X_d)_{tor}=0.\] Hence, we prove the theorem.
\end{proof}
\bibliographystyle{acm}
\def\cprime{$'$}

\end{document}